\documentclass[ a4paper,10pt]{article}
\pdfoutput=1
  
\usepackage[utf8]{luainputenc}
\usepackage[english]{babel}
\usepackage{csquotes}

\usepackage[plainpages=false,pdfpagelabels,hidelinks]{hyperref}

\title{Spectral accuracy for the Hahn polynomials}
\author{ René Goertz and Philipp Öffner}

\hypersetup{pdfauthor={Philipp Öffner}}
\hypersetup{pdftitle={Spectral accuracy for the Hahn polynomials }}
\hypersetup{pdfkeywords={hyperbolic conservation laws, spectral methods,
approximation theory, Hahn polynomials, discrete orthogonal polynomials, Fourier analysis, series expansion}}

\usepackage{amsmath}
\usepackage{amssymb}
\usepackage{commath}
\usepackage{mathtools}

\usepackage{siunitx}

\usepackage{amsthm}
\theoremstyle{plain}
\newtheorem{Definition}{Definition}[section]

\newtheorem{Satz}{Theorem}[section]
\newtheorem{Bemerkung}{Remark}[section]

\usepackage{graphicx}
\usepackage[small]{caption}
\usepackage{subcaption}

\usepackage{booktabs}
\usepackage{rotating}

\usepackage{xparse}

\NewDocumentCommand{\mat}{mo}{%
  \IfValueTF{#2}{%
    \underline{\underline{#1}}{#2}
  }{%
    \underline{\underline{#1}}\,
  }%
}

\renewcommand{\d}{\operatorname{d}}

\renewcommand{\L}{\mathbf{L}}

\renewcommand{\bf}{\textbf}
\renewcommand{\r}{\right}
\renewcommand{\l}{\left}

\renewcommand{\epsilon}{\varepsilon}
\renewcommand{\phi}{\varphi}
\newcommand{\N}{\mathbb{N}}
\newcommand{\R}{\mathbb{R}}

\begin{document}

\maketitle

\begin{abstract}

We consider in this paper the Hahn polynomials and their application in numerical methods. The Hahn polynomials are classical discrete orthogonal polynomials.
We analyse the behaviour of these polynomials in the context of spectral approximation of partial differential equations. 
We study series expansions $u=\sum_{n=0}^\infty \hat{u}_n \phi_n$, where the $\phi_n$  are the Hahn polynomials. 
We examine the Hahn coefficients and proof spectral accuracy in some sense. We substantiate our results by numericals tests. 
Furthermore we discuss a problem which arise by using the Hahn polynomials in the approximation of a function $u$, which is linked to the Runge phenomenon.
 We suggest two approaches to avoid this problem.
 These will also be the motivation and the outlook of further research in the application of discrete orthogonal polynomials in a spectral method for the numerical solution 
 of hyperbolic conservation laws. 
\end{abstract}

\section{Introduction}
Many numerical methods use series expansions of a function $u$ in  terms of  orthogonal polynomials, see  \cite{Canuto:01, glaubitz2016application, Ortleb:04, Oeffner:01, Schwab:01}
and references therein. For periodic problems the application of trigonometric polynomials is common, 
whereas Legendre, Chebyshev or generally Jacobi polynomials are usually applied for non-periodic problems.
All of these polynomials are  solutions of  singular Sturm Liouville problems and it is shown in  \cite{Canuto:01} that, 
if the functions $\phi_n$ fulfil a singular Sturm Liouville problem spectral convergence can be guaranteed,
i.e. the $n$-coefficient decays faster than every power of $n$ for an analytic function $u$. 
However, if the basis functions $\phi_n$ do not satisfy such a problem, then the coefficients in the expansion of a smooth function 
decay only with algebraic order.\\
The application we have in mind is the numerical solution  of hyperbolic conservation laws
\begin{equation*}
 u_t+f(u)_x=0
\end{equation*}
as they appear in numerical fluid dynamics and many other areas. 
The numerical methods we are dealing
with are the correction procedure via reconstruction (CPR) methods, also known under the name flux reconstruction (FR).
This method unifies several high order methods like the discontinuous Galerkin or spectral difference methods in a common framework,
for details see \cite{ huynh2007flux, huynh2014high,wang2009unifying,  ranocha2016summation},
but here it is sufficient to imagine the replacement of $u(x,t)$ by an expansion $U(x,t)=\sum_{n=0}^\infty \hat{u}_n(t) \phi_n(x)$ 
where the $\phi_n$ belong to certain classes
of orthogonal polynomials. In the approximation method we have to calculate the coefficients $\hat{u}_n(t)$. In literature
\cite{ Oeffner:01, Ortleb:05} two approaches 
can be found. One uses the classical {\bf{projection}}, where the coefficients $\hat{u}_k(t)$ are the Fourier coefficients. 
Therefore one solves 
\begin{equation}\label{Fourierkoeffizient}
 \hat{u}_n(t)=\frac{1}{||\phi_n||^2} \int_{\Omega} h(\xi) u(\xi,t)\phi_n (\xi)\; \d \xi,
\end{equation}
where $\Omega$ is the interval, $h$ is the weight function and $||\cdot||$ is the norm in the function space. In case of the 
Legendre polynomials $\Omega=[-1,1],
\;h(\xi)\equiv1$ and $||\cdot||$ is the usual $\L^2([-1,1])$-norm. \\
In numerics the calculation of the integral of \eqref{Fourierkoeffizient} is accomplished by quadrature, see \cite{Karniadakis:01}. 
One has to select the type of quadrature rule (Gauss- or 
Radau type) and also good  integration points usually the zeros of the basis functions $\phi_n$.\\
The second ansatz to calculate the coefficients   $\hat{u}_n$ is by using the {\bf{interpolation approach}}. 
One has to select ``good" interpolation points $x_j$ and must solve the 
linear equation system 
\begin{equation}\label{Interpolation}
 u(x_j,t)=\mathcal{V}(\hat{u}_n(t))
\end{equation}
with the Vandermonde matrix $\mathcal{V}=(\phi_n(x_j))_{n} $ in every time step. 
The system \eqref{Interpolation} has a unique solution,
if $\mathcal{V}$ is regular. Therefore we need the same numbers of 
interpolation points and basis elements $\phi_n$.
The position of the interpolation points has a massiv effect on the approximation and should also lead 
to good numerical properties of $\mathcal{V}$ like a small condition number. \\
Both approaches to compute the coefficients are not exact. Since we calculate these in every time step,
the numerical error  increases. \\
The idea of using discrete orthogonal polynomials is due to the projection approach. For continuous  
orthogonal polynomials one has to evaluate the integral in \eqref{Fourierkoeffizient}.
The discrete orthogonal polynomials come with a discrete scalar product and hence the integral becomes a sum.
By using discrete orthogonal polynomials we have only to compute this sum and
 the calculation of coefficients is exact. \\
This is only one reason for focussing on discrete orthogonal polynomials. Another is the
construction of discrete filters with the help of difference equations satiesfied by these polynomials. 
It is the same procedure as in \cite{ glaubitz2016application, Oeffner:01, glaubitz2016enhancing} with the difference that instead of using a differential operator 
one has to use a difference operator. 
However, in the context of this manuscript
we focus on the approximation results of the series expansion $\sum_{n=0}^m \hat{u}_n(t) \phi_n(x)$
and show spectral accuracy, if  $\phi_n$ are the Hahn polynomials. \\
The paper is organized as follows. The polynomials under consideration will be defined in 
the second section and some of their properties will be reviewed. 
Our main result  is Theorem \ref{Spactralaccu}, giving the decay of the coefficients of the expansion.
In section \ref{Numerics} we present  numerical test cases. 
Using discrete orthogonal polynomials as interpolants on an equidistant grid leads to  problems
which are equivalent to the Runge phenomenon. 
We present  possible solutions and finally 
conclude our results and give an outlook for further research.

\section{Hahn polynomials and their properties}\label{Polynome}

Here we introduce  the orthogonal polynomials under consideration, in particular we investigate the {\bf{Hahn polynomials}} in this paper. 
These are classical discrete orthogonal polynomials on 
an equidistant grid, which can be seen as the discrete analogue of the Jacobi polynomials. In the literature one may find two different definitions for the Hahn polynomials, for details 
see \cite{Lesky:01, Nikiforov:01}, 
one of which has a long history and already Chebyshev worked  with this definition, see \cite{Chebyshev:01}.
Here, we follow  the definition from \cite{Lesky:01, Olver:2010:NHMF}, 
which is common  nowadays.
\begin{Definition}\label{Hahn}
Let $N\in \N,\; -1<\alpha, \beta \in \R$ and  the intervall $I=[0,N]$ with a $(N+1)$-equidistant grid be given. The 
{\bf{Hahn polynomials}} are then defined as the hypergeometric function\footnote{For the definition of the hypergeometric function, see appendix \ref{Appendix} or \cite{Nikiforov:01}.} 
\begin{equation*}
  Q_n(x;\alpha, \beta,N):= 
  {}_3F_2\l(-n,n+\alpha+\beta+1,-x;\alpha+1,-N; 1  \r), 
\end{equation*}
where $n=0,\;,1\; \cdots,\; N.$
\end{Definition} 
For the sake of brevity we introduce the notation $Q_n(x):= Q_n(x;\alpha, \beta,N)$.\\
\begin{Bemerkung}
The Hahn polynomials are defined on the interval $I=[0,N]$. In the numerical tests  of section \ref{Numerics} we  transform the interval $I$ to $[-1,1]$
to have a comparison with the Legendre polynomials.
For the theoretical investigation we consider the Hahn polynomials on the interval $I$ and analyse the normalized Hahn polynomials $\tilde{Q}_n$ for simplicity. In 
principle the investigation of $Q_n$ on $[-1,1]$ is possible and leads to similar results. 
\end{Bemerkung}

Since we will later  need some well-known properties of the Hahn polynomials, we cite them from  \cite{Lesky:01}.
\begin{itemize}
 \item The Hahn polynomials are orthogonal on $I$ with respect to the inner product 
  \begin{equation}\label{Hahnortogonal}
 \begin{aligned} 
 <Q_n(x);Q_m(x)>_{\omega}:&= \sum\limits_{x=0}^NQ_n(x) Q_m(x) \omega(x)\\
 &=\sum\limits_{x=0}^NQ_n(x) Q_m(x) \binom{\alpha+x}{x} \binom{\beta+N-x}{N-x}\\
 &=\frac{(-1)^n(n+\alpha+\beta+1)_{N+1}(\beta+1)_nn!}{(2n+\alpha+\beta+1)(\alpha+1)_n(-N)_n N!}  \delta_{mn},  
 \end{aligned}
 \end{equation}
 for $m,n\in \N_0$ with $m,n\leq N$, where $\omega$ is the weight function given by
 \begin{equation}\label{Weightfunction}
   \omega(x):=\binom{\alpha+x}{x} \binom{\beta+N-x}{N-x}=\frac{\Gamma(\alpha+1+x) \Gamma(\beta+1+N-x)}{\Gamma(x+1)\Gamma(N+1-x)\Gamma(\alpha+1)\Gamma(\beta+1)}.
 \end{equation}

\item They satisfy the three term recurrence formula
\begin{equation}\label{Rekursionsformelhahn}
 -xQ_n(x)=A_nQ_{n+1}(x)-(A_n+C_n)Q_n(x)+C_nQ_{n-1}(x)
\end{equation}
with
\begin{align*}
A_n&=\frac{(n+\alpha+\beta+1)(n+\alpha+1)(N-n)}{(2n+\alpha+\beta+1)(2n+\alpha+\beta+2)}, \\
C_n&=\frac{n(n+\alpha+\beta+N+1)(n+\beta)}{(2n+\alpha+\beta)(2n+\alpha+\beta+1)}.
\end{align*}
\item The polynomials solve the eigenvalue equation
\begin{equation}\label{Hahndifferenzengleichung}
\lambda_nQ_n(x)=B(x)Q_n(x+1)-[B(x)+D(x)]Q_n(x)+D(x)Q_n(x-1)
\end{equation}
with eigenvalues $\lambda_n=n(n+\alpha+\beta+1)$,\\
$B(x)=(x+\alpha+1)(x-N)$ and $
 D(x)=x(x-\beta-N-1).$
By using the difference operators 
\begin{align*}
 \Delta f(x):=f(x+1)-f(x),\\
 \nabla f(x):=f(x)-f(x-1),
\end{align*}
and their identities
\begin{align}
 \Delta f(x)&=\nabla f(x+1), \label{identitatvor} \\
 \Delta \l[ f(x)g(x)\r]&=f(x)\Delta g(x)+g(x+1)\Delta f(x), \label{identitatgem} \\
 \nabla \l[f(x)g(x) \r]&= f(x-1)\nabla g(x)+g(x)\nabla f(x), \label{identitatgemzwei}
\end{align}
we can reshape equation \eqref{Hahndifferenzengleichung} in the following self-adjoint form
\begin{equation}\label{Hahnselbstadjungiert}
 \Delta [-D(x)\omega(x)\nabla Q_n(x)]+\lambda_n(x) \omega(x) Q_n(x)=0,
\end{equation}
with weight function $\omega$.
 \end{itemize}

 \section{Spectral accuracy}\label{Spectral}
 Spectral accuracy/convergence  means that the $n-$th coefficient in the expansion of a smooth function 
 decays faster to zero than any power of $n$. Spectral convergence is something like the Holy Grail in Numerical Analysis and many
 numerical methods are designed to exploit this type of convergence. 
 
 In this section we analyse the behaviour of the Hahn coefficients.
 It is misleading to speak in this context about spectral convergence, because all coefficients $\hat{u}_n$ are equal to zero for  $n>N$ or not defined.
 
Due  to this fact we however speak  about  spectral accuracy
in the following sense:
Accuracy is called spectral, if an index $n_1$ exists so that the absolute values of the $\hat{u}_n$, $N \geq n > n_1$,
decrease faster than any power of $n$. Although we consider the interval $I=[0,N]$ with an $(N+1)$-equidistant grid. The transformation to any compact intervall $[a,b]$
is possible and follows analogously.

  From  Theorem \ref{Spactralaccu}  spectral accuracy follwos directly for Hahn polynomials. 

\begin{Satz}\label{Spactralaccu}
Let $\alpha, \beta>-1,\;,m,\;N\in \N$ with $m\leq N$, $I=[0,N]$ and $u\in C^{\infty}([-1,N+1])$. 
$\tilde{Q}_n(x,\alpha,\beta,N)$ are the normalized Hahn polynomials of degree $n \leq N$. 
The Hahn projection of $u$ with degree $m$ is given by  
 \begin{equation*}
  P_m u(x)=\sum\limits_{n=0}^m \hat{u}_n \tilde{Q}_n(x),
 \end{equation*}
with the coefficient	\[
                          \hat{u}_n= <\tilde{Q}_n;u>_{\omega}
                         \]
and  weight function $\omega$.
It holds
\begin{equation*}
 |\hat{u}_n|\leq \frac{1}{n^{2k}}  \l(  \sum\limits_{i=0}^N \omega(i) \l( \L^k_{disk}u(i)\r)^2 \r)^\frac{1}{2}
\end{equation*}
for all $k\in \N_0$, where   $\L_{disk}:=\frac{1}{\omega(i)}\Delta[ -D(i)\omega(i)\nabla ] $ is the discrete difference operatator. 
\end{Satz}

 \begin{proof}
 We have 
  \begin{equation*}
  \hat{u}_n=<\tilde{Q}_n;u>_{\omega}=
  \sum\limits_{i=0}^N \omega(i)u(i)\tilde{Q}_n(i).
 \end{equation*}
 Using equation \eqref{Hahnselbstadjungiert} we get 
\begin{equation*}
 \hat{u}_n= \sum\limits_{i=0}^N \omega(i)u(i)\tilde{Q}_n(i) =\frac{-1}{ \lambda_n } \sum\limits_{i=0}^N u(i) \Delta \l[-D(i)\omega(i)\nabla \tilde{Q}_n(i) \r],
\end{equation*}
where $\omega(N+1) \equiv 0$ and $D(0)\equiv 0 $. We employ summation by parts 
\begin{equation*}
 \sum\limits_{i=0}^{N} f(i)\Delta g(i) =f(i)g(i)\Bigg|_0^{N+1}-\sum\limits_{i=0}^{N} g(i+1)\Delta f(i),
\end{equation*}
and obtain 
\begin{align*}
 \hat{u}_n&=\frac{-1}{\lambda_n } \sum\limits_{i=0}^N u(i) \Delta \l[-D(i)\omega(i)\nabla \tilde{Q}_n(i) \r]\\
          &= \frac{-1}{ \lambda_n } \Bigg(\underbrace{-u(i)D(i)\omega(i)\nabla \tilde{Q}_n(i)\Bigg|_0^{N+1}}_{=0} -\sum\limits_{i=0}^N \Delta u(i)
          ( -D(i+1)\omega(i+1)\nabla \tilde{Q}_n(i+1) )  \Bigg).
\end{align*}
The identity \eqref{identitatvor} and summation by parts yield  
\begin{align*}
 \hat{u}_n&=\frac{1}{ \lambda_n } \l(\sum\limits_{i=0}^N  \nabla \tilde{Q}_n(i+1) \l( -D(i+1)\omega(i+1) \nabla u\l(i+1\r) \r)   \r)\\
  &= \frac{1}{\lambda_n }  \Bigg( -\tilde{Q}_n(i) D(i)\omega(i) \nabla u(i) \Bigg|_{i=0}^{N+1} -\sum\limits_{i=0}^N \tilde{Q}_n(i) \nabla \l[ -D(i+1)\omega(i+1) \nabla u\l(i+1\r)\r] \Bigg) \\
  &=\frac{-1}{ \lambda_n }  \sum\limits_{i=0}^N \omega(i) \tilde{Q}_n(i) \frac{1}{\omega(i)} \Delta \l[ -D(i)\omega(i) \nabla u(i)\r]\\
 & =\frac{-1}{ \lambda_n}  \l( \sum\limits_{i=0}^N \omega(i) \tilde{Q}_n(i) \underbrace{\frac{1}{\omega(i)} \Delta \l[ -D(i)\omega(i) \nabla u(i)\r]}_{=\L_{disk}} \r).
\end{align*}
Applying this procedure $k$ times it follows
\begin{equation*}
 \hat{u}_n=\frac{(-1)^k}{\lambda_n^k}  \l( \sum\limits_{i=0}^N \omega(i) \tilde{Q}_n(i) \L_{disk}^k u(i)\r).
\end{equation*}
We consider the absolute value  $|\hat{u}_n|$ and use the Schwarz inequality to get
\begin{equation*}
|\hat{u}_n| \leq \frac{1}{ \lambda_n^k} ||\tilde{Q}_n||_{\omega} \l(  \sum\limits_{i=0}^N \omega(i) \l( \L^k_{disk}u(i)\r)^2 \r)^\frac{1}{2}
 \approx \frac{1}{ n^{2k}}  \l(  \sum\limits_{i=0}^N \omega(i) \l( \L^k_{disk}u(i)\r)^2 \r)^\frac{1}{2}.
\end{equation*}
The sum is well-defined and independent of $n$ for fixed $N$ and $u\in C^\infty$.
The decay behaviour of $\hat{u}_n$ is characterized by  $\frac{1}{n^{2k}}$ for $n\leq N$. 
If  $n>N$ then   $\hat{u}_n $ is equal to zero. 
\end{proof}
\begin{Bemerkung}
The value $ \l(  \sum\limits_{i=0}^N \omega(i) \l( \L^k_{disk}u(i)\r)^2 \r)^\frac{1}{2}$ may be problematic, but for any $u \in C^\infty$ it is well-defined and bounded. 
Nevertheless we should be careful with this term, especially for small $N$. In this case the truncated series expansion may not describe the function qualitatively. 
Hence, we will also investigate the error between a function $u$ and the  truncated Hahn expansion of $u$. 
In our investigation \cite{Goertz:01} we analyse the error under the assumption of  the two limit processes $n \to \infty$ and $N\to \infty$.
\end{Bemerkung}

\section{Numerical tests} \label{Numerics}
After presenting our theoretical results we give a short numerical investigation to show that our conclusions are justified.
In our first example we approximate the function $f(x):=\sin \l( \pi\cdot x \r)$ in the intervall $I=[-1,1]$ by a truncated Hahn series. 

We use $N=30$ , $(\alpha,\beta) =(0,0)$ (red), $(0.5,0.5)$ (green), ($5,0)$ (blue) and expand the series up to $m=10$. 
In figure \ref{approximitionerror1} 
we see the truncation error $ f(x)-\sum\limits_{n=0}^{10} \hat{f}_n \tilde{Q}_n(x)$ over the intervall $I$. 
\begin{figure}
 \includegraphics[width=1\textwidth]{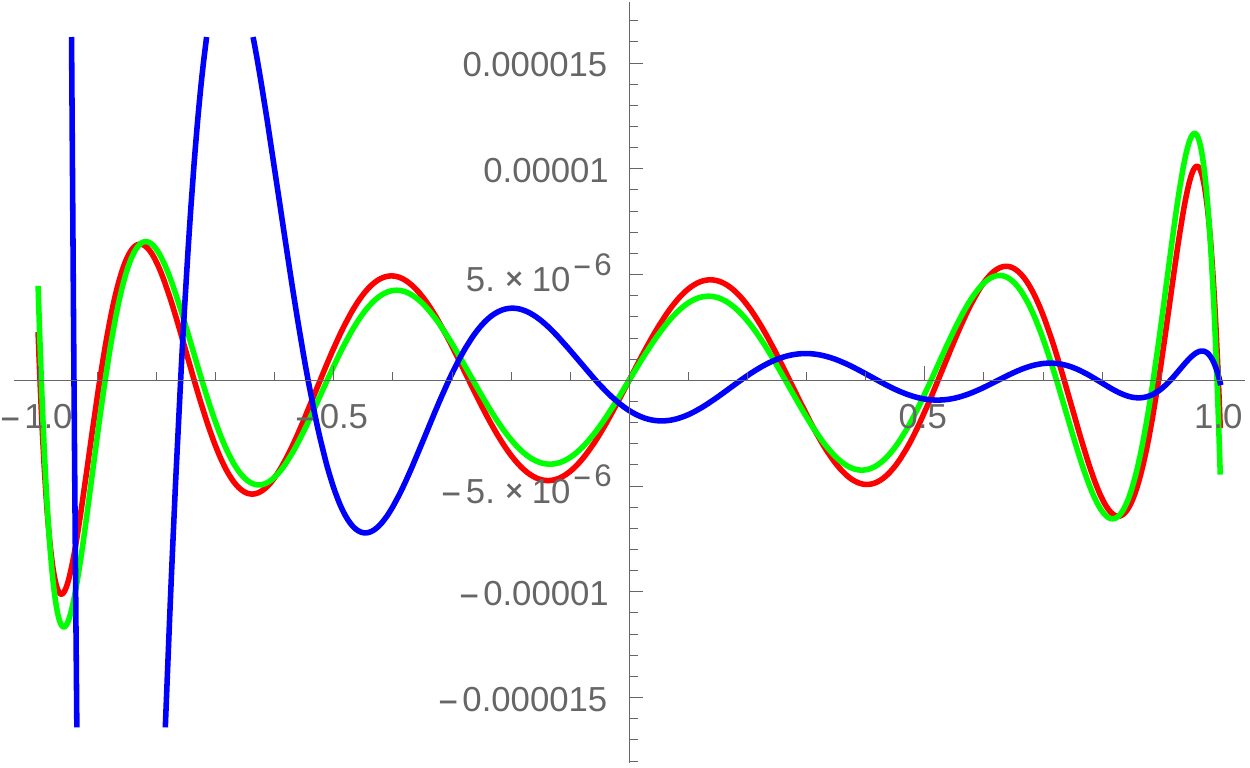}
 \caption{ $N=30, m=10$, approximation error, parameters: $(0, 0)$ (red), $(0.5, 0.5)$ (green), $(5, 0)$ (blue)}\label{approximitionerror1}
\end{figure}
We recognize that for all parameter selections the pointwise error is very small and the truncated series expansions describe the function $f$ fairly good.
This fact is also reflected in the coefficients  $\hat{f}_n$ the absolute values of which are shown in  Figure \ref{approximitioncoefficientsinus1}.
\begin{figure}[ht]
 {\includegraphics[width=0.5\textwidth]{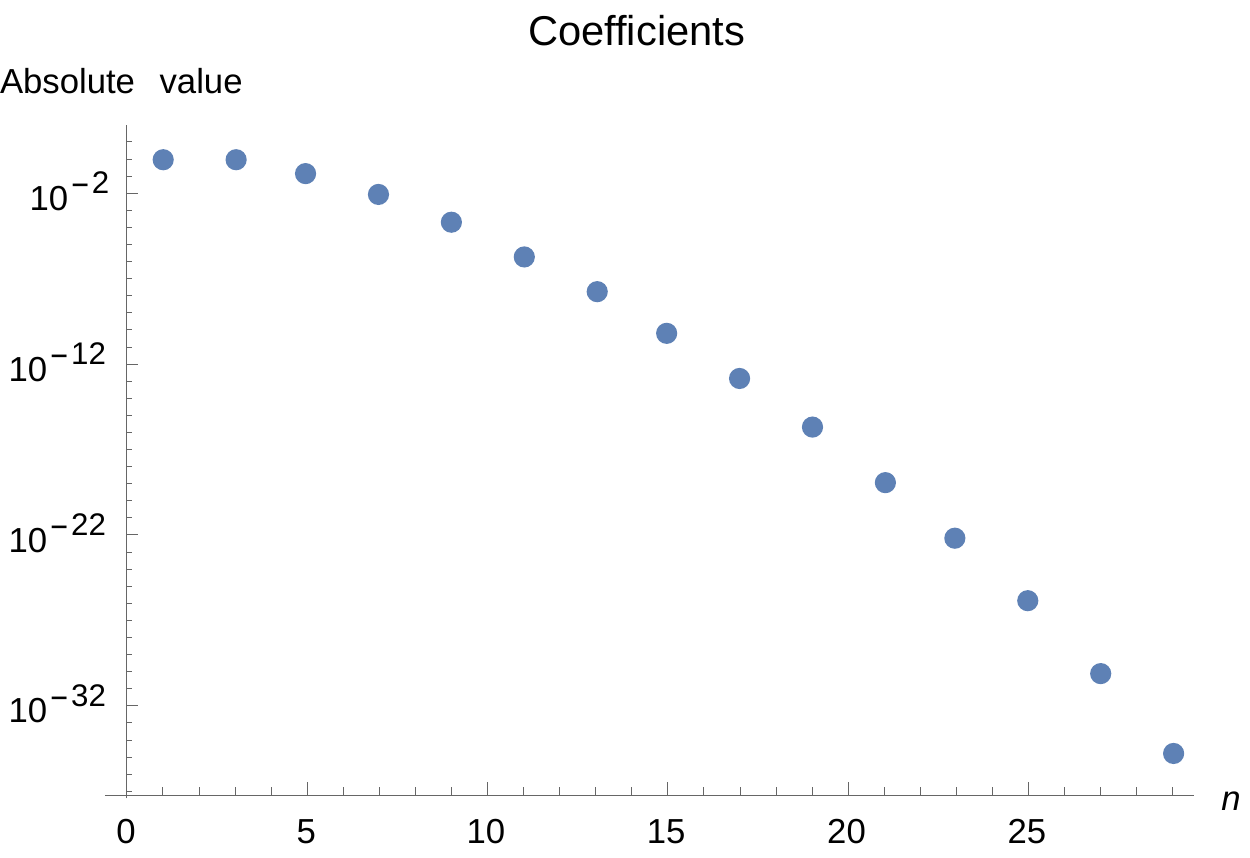}} 
 {\includegraphics[width=0.5\textwidth]{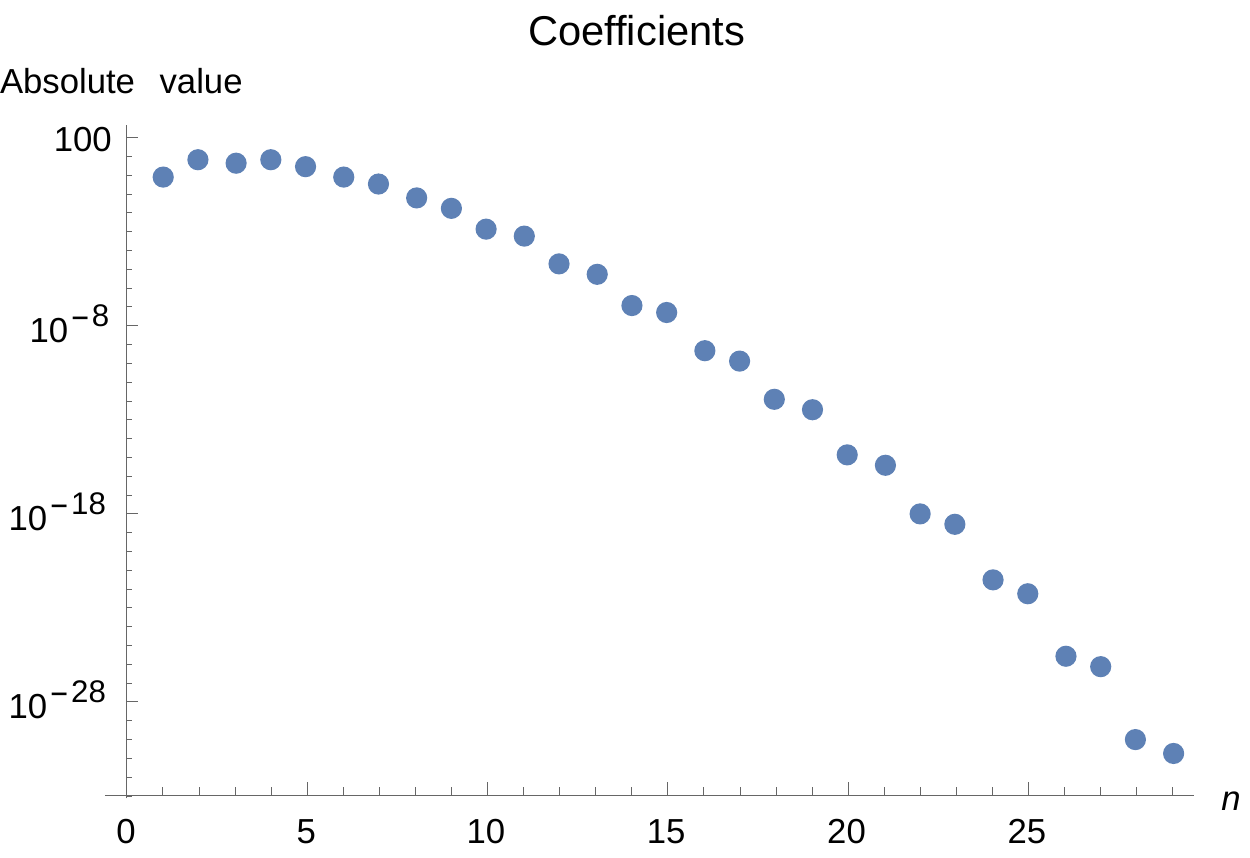}}
 \caption{Absolute values of the coefficients $\hat{f}_n$ }\label{approximitioncoefficientsinus1}
\end{figure}
On the left hand side the Hahn polynomials with  the parameters $(0,0)$ are used and on the right-hand side $(5,0)$.
We realize  an exponential decay in logarithmic scale in both figures. Furthermore, for $(0,0)$  the odd coefficients only appear. 
All even coefficients are  zero and have no influence on the approximation. Because of the symmetry of the sine function
and the special choice of the parameters $(\alpha, \beta)=(0,0)$ this is not suprising. 
We see this same effect in the Legendre coefficients\footnote{We use for approximation the classical Legendre polynomials in the series expansion.},
where all even coefficients are zero, as can be  seen in figure \ref{Legendrecoefficient}. Another interesting fact is the approximation speed being faster by using Hahn polynomials,
compare figures \ref{Legendrecoefficient} and \ref{approximitioncoefficientsinus1}.
\begin{figure}
 {\includegraphics[width=0.5\textwidth]{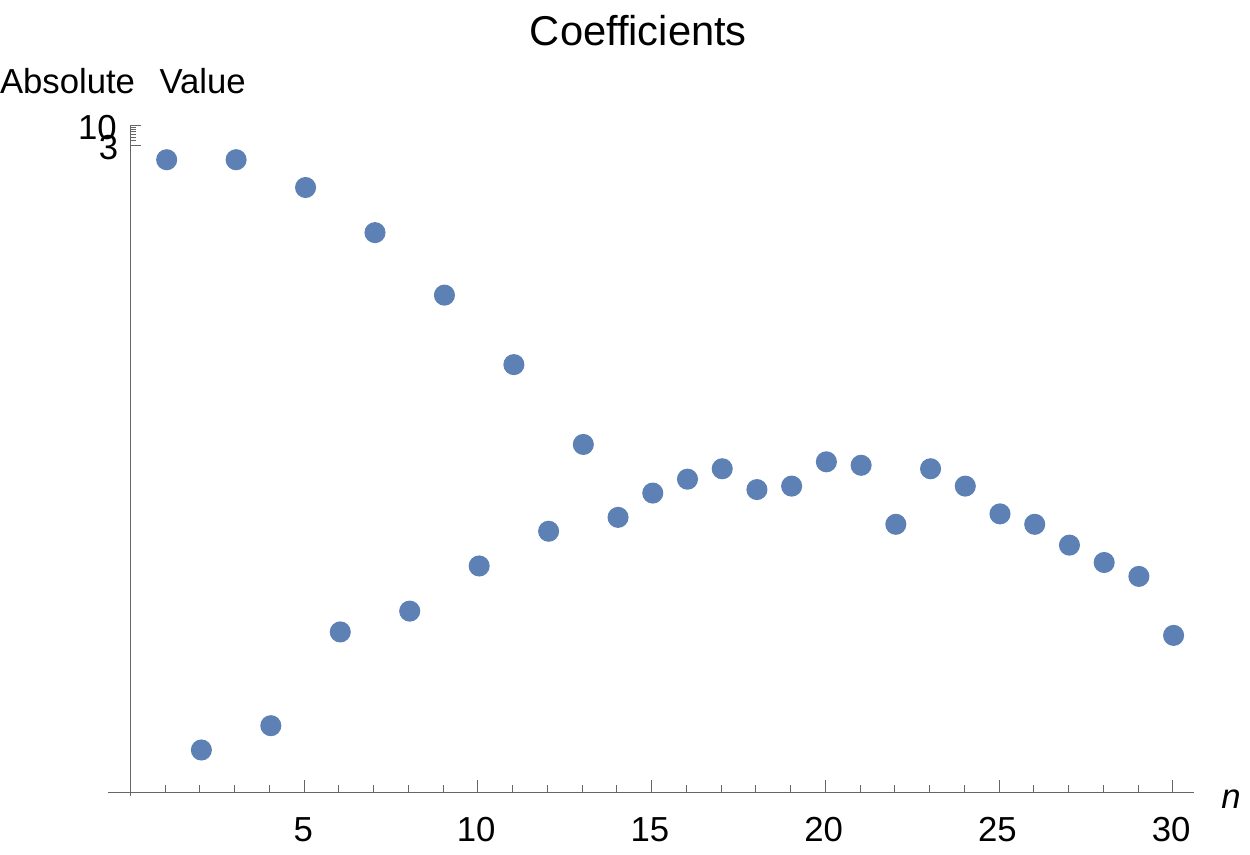}} 
 {\includegraphics[width=0.5\textwidth]{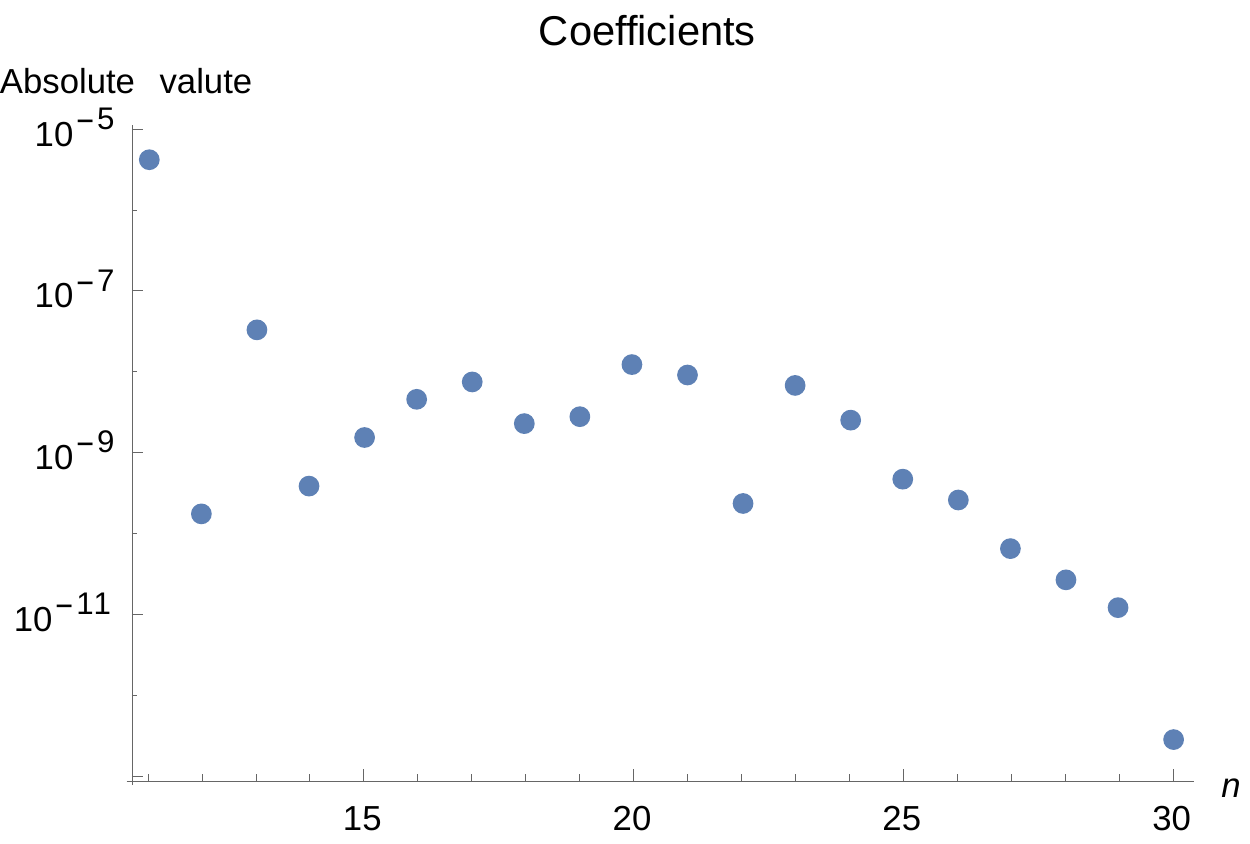}}
 \caption{Absolute values of the coefficients $\hat{f}_n$ }\label{approximitioncoefficientsinus2}
\end{figure}
Figure \ref{approximitioncoefficientsinus2} demonstrates the absolute values of the coefficients in the case  $(0.5,0.5)$.
Taking a closer look on the behaviour of the coefficients, we see that the 
even  coefficients do not play a major role in the approximiaton, but their influence rises with growing $n$ and their absolute value increases.
The  value of the even and odd coefficients is nearly the same around $n=14/15$ and remain up to  $n=23$ at this level which is  between $[10^{-10},10^{-8}]$.
This circumstances can be explained by the fact that the Hahn polynomials of even degree like,  $Q_{2n}(x,\alpha, \beta,N)$, have odd parts.
When $n$ increases also these parts grow and their influence grows. 
However, in  the end we again get spectral accuracy in the coefficients, what was predicted in  Theorem \ref{Spactralaccu}. 
\begin{Bemerkung}
  We note that this special behaviour of the rising influence of the even coefficients is characteristic for using Hahn polynomials with parameters $\alpha, \beta \notin \N_0$
  in approximating the odd sine function. A comparable result can be seen if approximating an even function. Then the influence of the odd coefficients rises at the beginning.
\end{Bemerkung}
\begin{figure}[ht]
 {\includegraphics[width=0.5\textwidth]{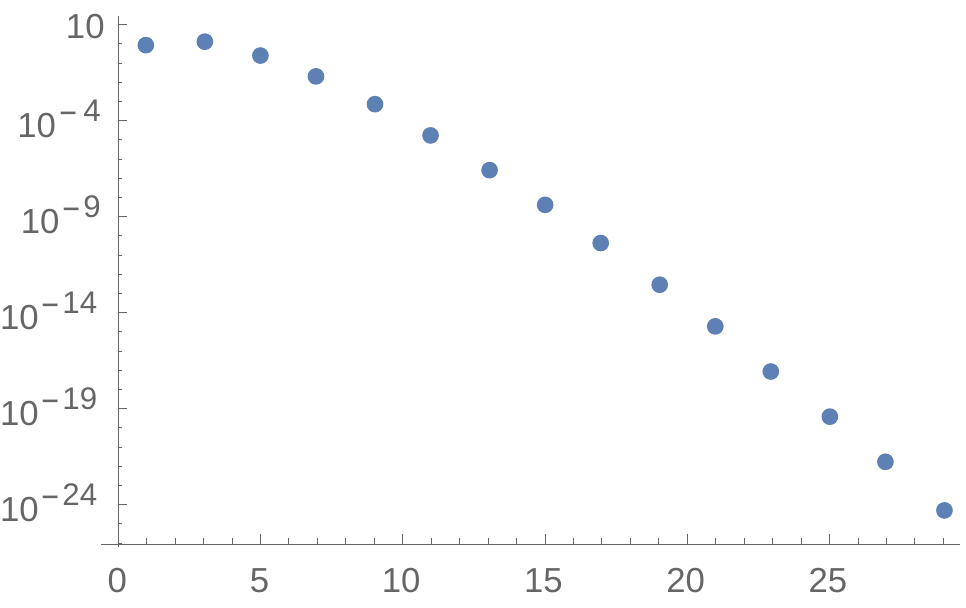}} 
 {\includegraphics[width=0.5\textwidth]{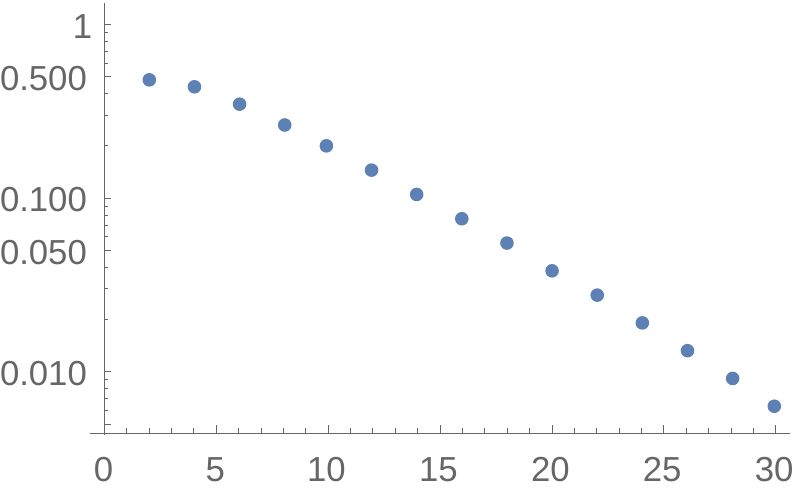}}
 \caption{Absolute values of the coefficients for  $\hat{f}_n$ and $\hat{g}_n$ using Legendre polynomials}\label{Legendrecoefficient}
\end{figure}
In the second example we consider the function  $g(x)=\frac{1}{1+25x^2}$ on the intervall $I=[-1,1]$. 
The absolute value of the Hahn coefficients are plotted in a logarithmic scale. We use  $N=30$ and  the parameters  $(\alpha,\beta)= (0,0),\; (5,0)$ and $(0.5,0.5)$.
\begin{figure}
 {\includegraphics[width=0.5\textwidth]{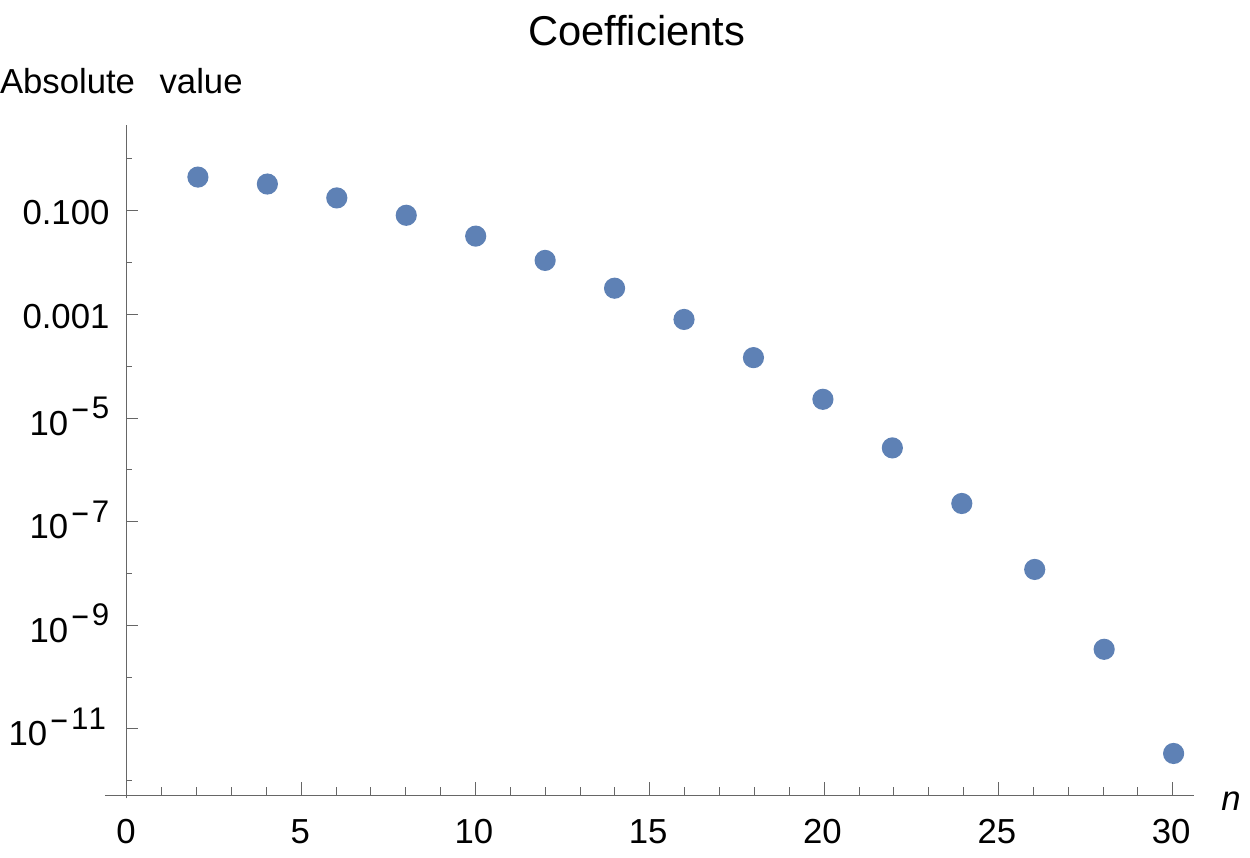}} 
 {\includegraphics[width=0.5\textwidth]{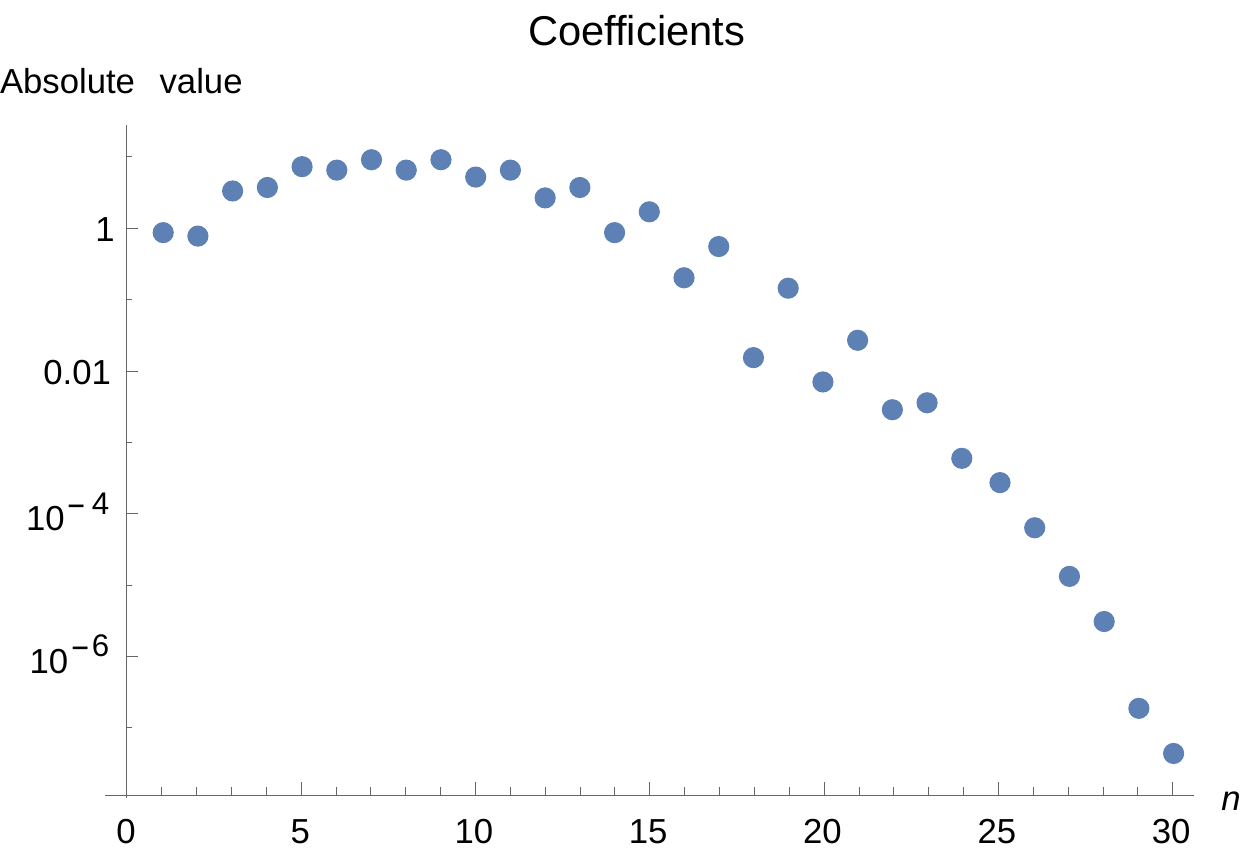}}
 \caption{Absolute values of the coefficients $\hat{g}_n$, parameters $(0,0)$ and $(5,0)$ }\label{approximitioncoefficientrunge}
\end{figure}
In figures \ref{approximitioncoefficientrunge} and  \ref{approximitioncoefficientrunge2} we see the coefficients $\hat{g}_n$ as functions of  the degree $n$. 
We recognize a behaviour, which we have already noticed in the first example. In the case of the parameters  $ (\alpha, \beta)=(0,0)$ 
all odd coefficients are zero and only the even coefficients are needed in the approximation. 
In case of  $(0.5,0.5)$ the absolute value of the odd coefficients increases  due to the fact that the Hahn polynomials of odd degree 
have an even part, and this part grows in $n$.

\begin{figure}
 {\includegraphics[width=0.5\textwidth]{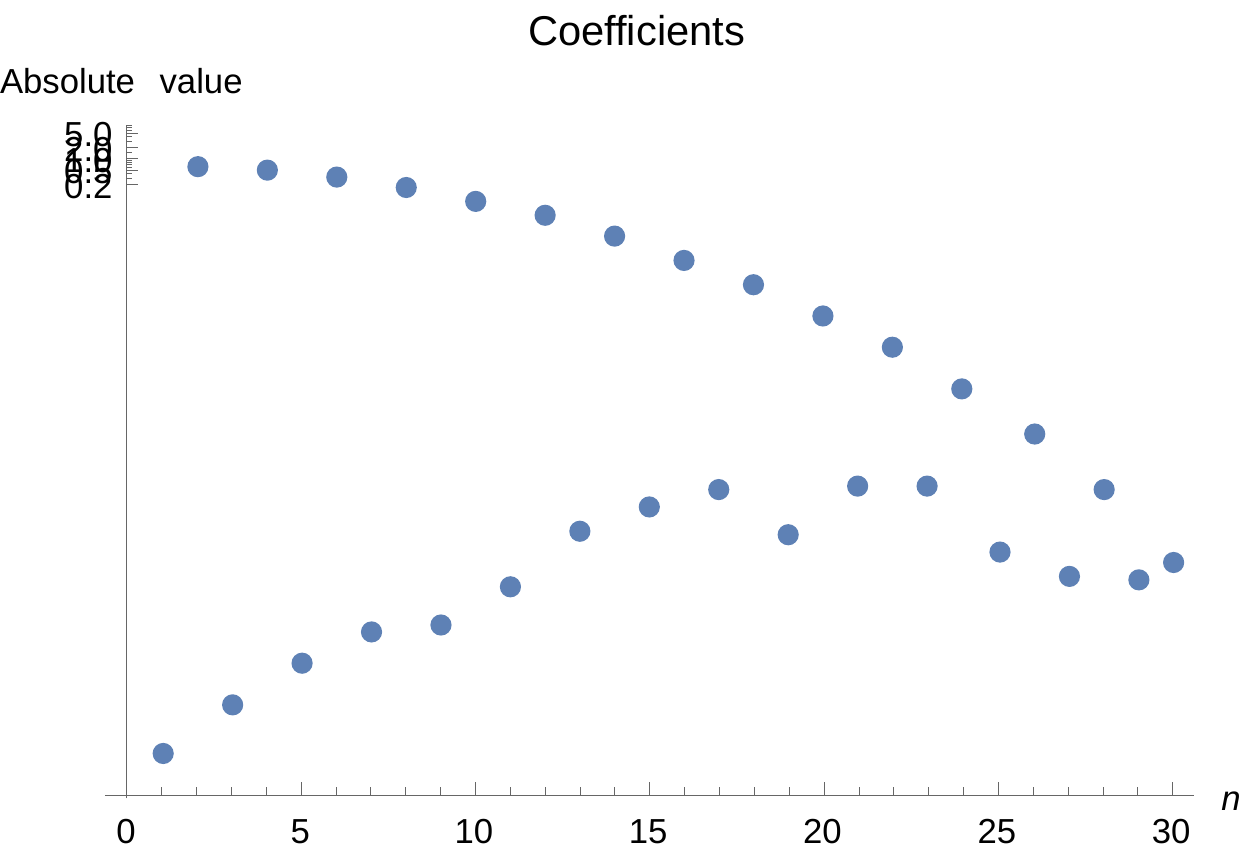}}
  {\includegraphics[width=0.5\textwidth]{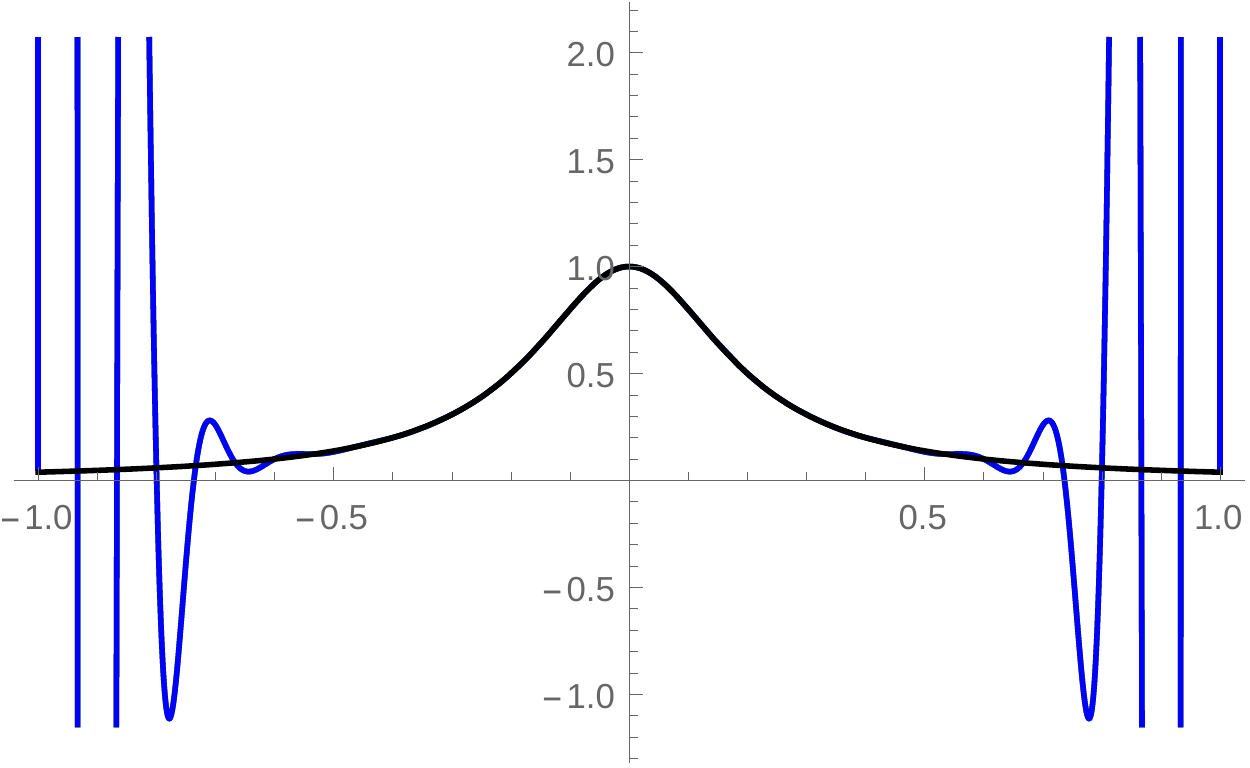}} 
 \caption{Absolute values of the coefficients $\hat{g}_n$ parameters $(0.5,0.5)$, Runge phenomenon}\label{approximitioncoefficientrunge2}
\end{figure}

In all cases we see spectral accuracy in the coefficients even though the absolute values increase in the beginning  in the case $(5,0)$.
Simultaneously we do not get the same level as in the sine example and also we get a  further problem.
It is well-known that approximating  the function $g$ with the usual interpolation polynoms on an equidistant grid yields  the 
Runge phenomenon. If we approximate with $m=30$ our truncated Hahn series becomes the usual  interpolation polynomial 
on equidistant points and we also obtain the Runge phenomenon, see figure \ref{approximitioncoefficientrunge2}, where the black line marks the original function $g$.  
In figure \ref{runge} we plot  the truncated Hahn series with $m=10$ for the parameters $(\alpha,\beta) =(0,0)$ (red), $(0.5,0.5)$ (green), ($5,0)$ (blue).
In the right figure we demonstrate the pointwise  approximation error. We conclude that all truncated series expansions do not 
describe the original function and the quality of the approximation also depends on the parameter selection  $\alpha, \beta$.
The Runge phenomenon has an influence on the approximation and we have to take this fact into  consideration.
\begin{figure}
 {\includegraphics[width=0.5\textwidth]{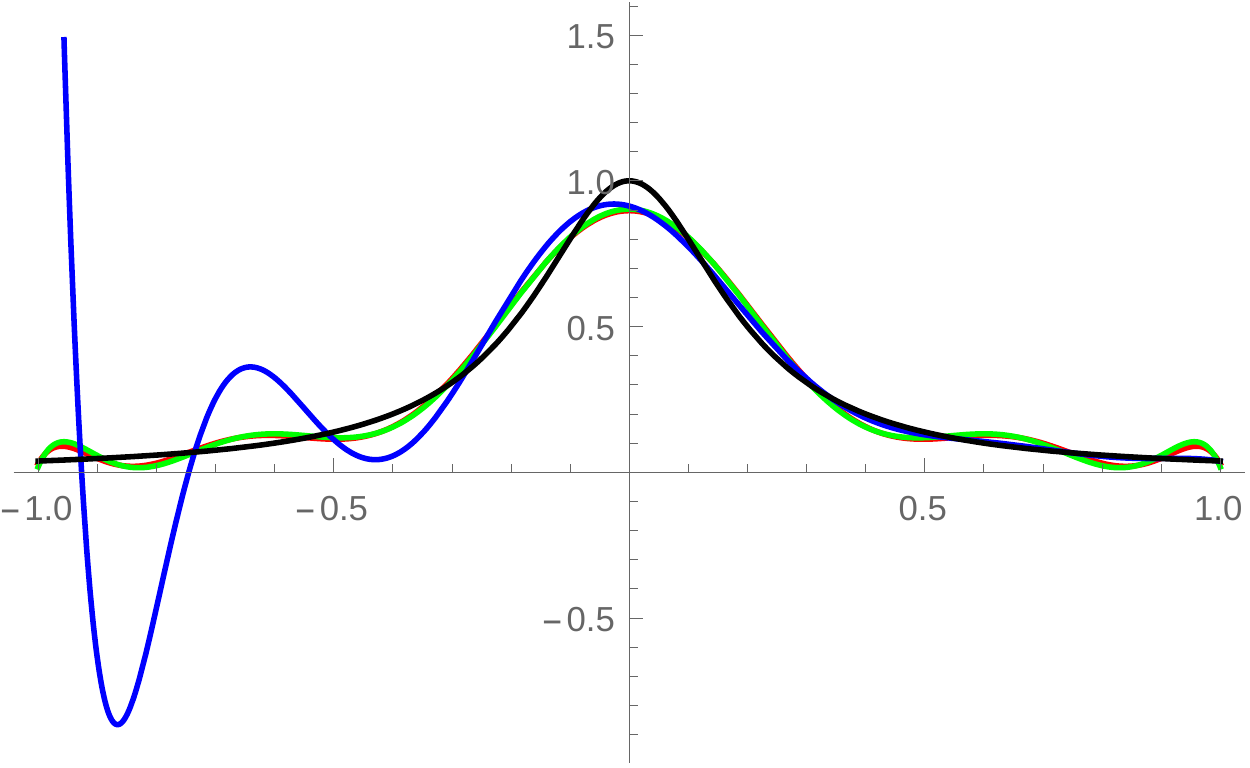}}
  {\includegraphics[width=0.5\textwidth]{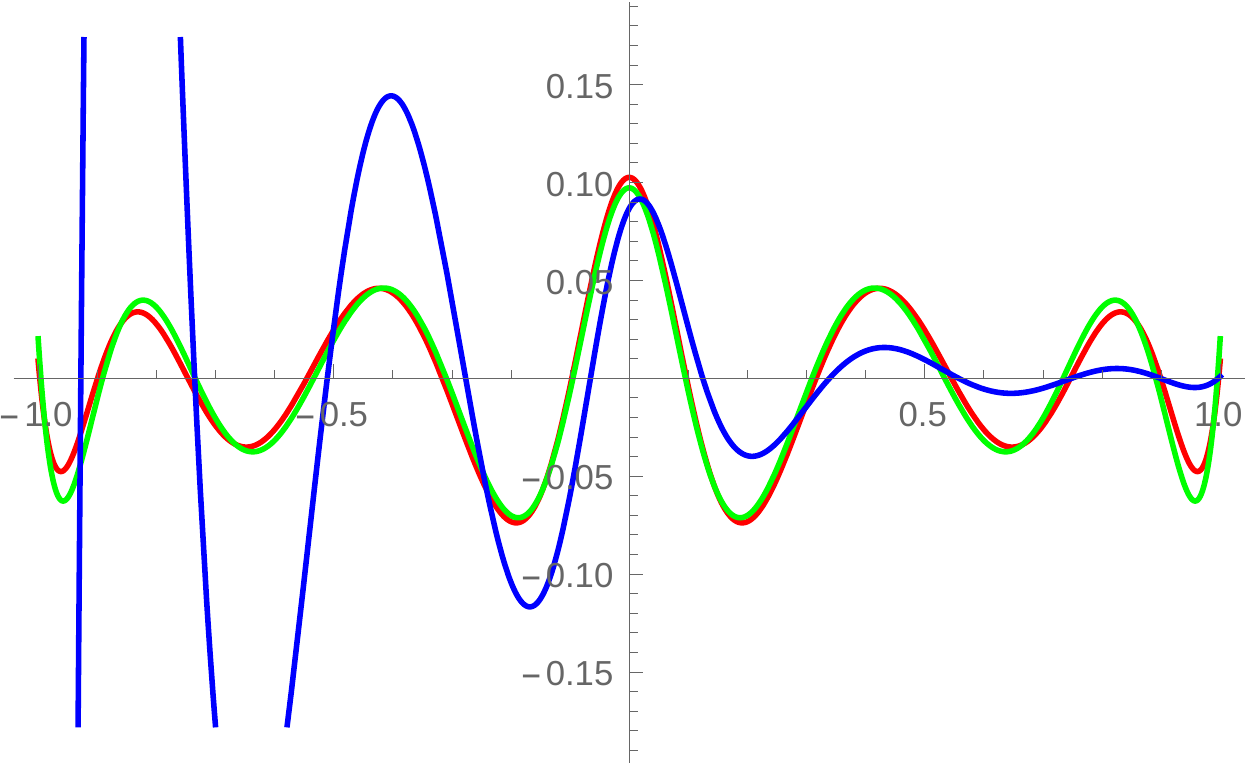}}
 \caption{Truncated series with $m=10$, approximation error }\label{runge}
\end{figure}

\section{Discussion and an outlook}
Applying the Hahn polynomials in a spectral method for the numerical solution of conservation laws we will not commit
any numerical calculation errors by using the projection approach, but we may have also to deal with 
the Runge phenomenon. We can not ensure that our approximated solution really describes the correct one and the
question  is now: How can we fix this problem?\\ 
In \cite{Gelb:03} the authors use discrete orthogonal polynomials for approximation and solving differential equations. 
Their  basic idea is to split the domain and to approximate  not on equidistant points near the boundary, instead they use Chebyshev nodes to avoid the Runge phenomenon.

In \cite{Eisinberg:01} the authors construct  discrete orthogonal polynomials via a quadrature rule from Chebyshev nodes up to a particular, but fixed degree.
These are not classical discrete orthogonal polynomials and therefore some basic properties  do not exist  or are not known like the satisfaction of an eigenvalue equation.
Nevertheless one should analyse their approximation properties and also apply them in a numerical method to solve differential equations. 
We will consider this in a further paper.\\
Here our focus was on the Hahn polynomials and their approximation properties, because they have all the basic properties which we will need later 
to construct a discrete filter. Furthermore we can extend these polynomials similar to the Jacobi polynomials in the continuous
case to triangular meshes. Many numerical methods use triangulations to split the domain into different elements. In the area of computational fluid dynamics triangular grids  
are optimal for the description of complex geometries. This is one reason why we are interested in an extension of  the Hahn polynomials on triangles, and  
this was already done by Xu, see \cite{Xu:03, Xu:02}. \\
First of all, we have to deal with the problem of the Runge phenomenon, when using the truncated Hahn series. Therefore,
we have to understand the approximation properties of the series
and the pointwise error in particular. We already started to investigate the behaviour of the pointwise truncation error in \cite{Goertz:01} and  
estimated the pointwise error of the truncated Hahn series with respect to the pointwise error of the truncated Jacobi series.
The occurrence of the Runge phenomenon is already studied in-depth, see \cite{trefethen2013approximation}, but nevertheless, we are also 
looking for  additional and further conditions on the function $u$ to  ensure that  the Runge phenomenon has no influence on the approximation.\\
The last idea is motivated by the Gegenbauer reconstruction in  \cite{Shu:01, GOTTLIEB199281}. 
The basic idea of the Gegenbauer reconstruction is to develop the numerical solution once more in a series expansion 
to delete the Gibbs' phenomenom. The authors investigate the ratio between the order  $m$ of the Gegenbauer polynomials and the truncation index $N$ of the series.
They  prove some convergence results, for details see \cite{Shu:01}. In \cite{Boyd:01} the author showed that singularities off the real axis can destroy convergence
if $m,N$ tend simultaneously to infinity. The Gegenbauer reconstruction must therefore be constrained to use a sufficiently small ratio of order $m$ to truncation $N$.
In \cite{Gelb:02} the reconstruction procedure is further investigated using another orthogonal basis of the Freud polynomials.
We will transfer their ideas to our problem and will analyse  the ratio between the $N$ Hahn polynomials and truncation $m$, which are finally used in the series expansion,
to minimize the total error. 

\section{Appendix:}\label{Appendix}
Let $s,r\in \N$. 
The hypergeometric function  ${}_rF_s$ is defined by the series
\begin{equation}
   {}_rF_s\l(a_1,\cdots,a_r;b_1,\cdots b_s; z  \r):=\sum\limits_{k=0}^\infty \frac{(a_1,\cdots,a_r)_k}{(b_1,\cdots, b_s)_k} \frac{z^k}{k!}, 
\end{equation}
where
\begin{equation*}
 (a_1\cdots,a_r)_k:=(a_1)_k\cdots(a_r)_k.
\end{equation*}
The parameters must be such that the denominator factors in the terms of the series are never zero. When one of the numerator parameters $a_i$ equals $-n$, 
where $n$ is a nonnegative integer, this hypergeometric function is a polynomial in $z$. 

\bibliographystyle{plain}
\bibliography{literatur}

\end{document}